\documentclass[12pt]{amsart}
\usepackage{enumerate}
\usepackage{url}
\usepackage{xspace}
\usepackage{graphicx}

\usepackage[margin=1in]{geometry}

\DeclareMathOperator{\pnt}{\raise 0.5mm \hbox{\large\textbf{.}}}

\newcommand{\note}[2][ ]{}

\newtheorem{theorem}{Theorem}[section]
\newtheorem{lemma}[theorem]{Lemma}

\newtheorem{corollary}[theorem]{Corollary}

\theoremstyle{definition}
\newtheorem{definition}[theorem]{Definition} 
\newtheorem{remark}[theorem]{Remark}
\newtheorem{example}[theorem]{Example}

\usepackage{color,hyperref}
\definecolor{darkblue}{rgb}{0,0,0.6}
\hypersetup{colorlinks,breaklinks,
            linkcolor=darkblue,urlcolor=darkblue,
            anchorcolor=darkblue,citecolor=darkblue}

\title[The KOH terms and classes of unimodal $N$-modular diagrams]{The KOH terms and classes \\of unimodal $N$-modular diagrams} 

\author{Fabrizio Zanello} \address{Department of Mathematics\\ MIT\\ Cambridge, MA 02139-4307, {\tiny and}}
\address{Department of Mathematical  Sciences\\ Michigan Tech\\ Houghton, MI  49931-1295}
\email{zanello@math.mit.edu}

\begin{document}


\begin{abstract}
We show how certain suitably modified $N$-modular diagrams of integer partitions provide a nice combinatorial interpretation for the general term of Zeilberger's KOH identity. This identity is the reformulation of O'Hara's famous proof of the unimodality of the Gaussian polynomial as a combinatorial identity. In particular, we determine, using different bijections, two main natural classes of modular diagrams of partitions with bounded parts and length, having the KOH terms as their generating functions. One of our results greatly extends recent theorems of J. Quinn \emph{et al.}, which presented striking applications to quantum physics.
\end{abstract}

\maketitle

\section{Introduction} 

The \emph{Gaussian polynomial} $\binom {n}{k}_q$ is the $q$-analogue of the binomial coefficient $\binom {n}{k}$. It is defined as
$$\binom {n}{k}_q:= \frac{(1-q)(1-q^2)\cdots (1-q^n)}{(1-q)(1-q^2)\cdots (1-q^k)\cdot (1-q)(1-q^2)\cdots (1-q^{n-k})}.$$
Gaussian polynomials play an important role in several fields of mathematics, including bijective combinatorics and  partition theory. Most importantly for us here, $\binom {n}{k}_q$ is the generating function for all integer partitions whose Ferrers diagrams are contained inside an $(n-k) \times k$ rectangle. It is well known that Gaussian polynomials are unimodal and symmetric about $k(n-k)/2$. D. Zeilberger's KOH Theorem \cite{Z1,Z2} gives a beautiful insight into the combinatorics of K. O'Hara's celebrated constructive proof \cite{OH} of the unimodality of Gaussian polynomials; in particular, the KOH Theorem decomposes a Gaussian polynomial into a finite sum of suitable polynomials, all unimodal, with nonnegative integer coefficients, and symmetric about the same degree. 

Given a nonnegative integer $n$, we say that the weakly decreasing sequence $\lambda=(\lambda_1,\lambda_2,\dots )$ of nonnegative integers is a \emph{partition} of $n$, and sometimes write $\lambda\vdash n$, if $\sum_{i\geq 1}\lambda_i=n$. Set $Y_i:= \sum_{j=1}^i\lambda_j$ for all $i\geq 1$, and $Y_0:=0$. We have:

\begin{theorem}[KOH]\label{koh}
$$\binom{a+b}{a}_q=\sum_{\lambda\vdash b} q^{2\sum_{i\geq 1}\binom{\lambda_i}{2}} \prod_{j\geq 1} \binom{j(a+2)-Y_{j-1}-Y_{j+1}}{\lambda_j-\lambda_{j+1}}_q.$$
\end{theorem}

(See \cite{Z1,Z2}, and also \cite{Ma0} for an elementary algebraic proof of the KOH Theorem.) It can be seen that all terms of the sum in the right-hand side of the KOH identity are  unimodal polynomials in $q$, with nonnegative  coefficients, and symmetric about $ab/2$. A very interesting problem is, therefore, to find natural combinatorial interpretations of the KOH summands as generating functions for suitable classes of partitions (which are, as a consequence, rank-unimodal and rank-symmetric). Of course, notice that the KOH summands, at least implicitly, already have a combinatorial meaning, because of how they have been derived in the first place --- by ``algebraizing'' a combinatorial proof.

The summand being contributed by $\lambda=(b,0,0,\dots)$ was studied in \cite{BQQW}; later, more generally, the terms corresponding to all partitions $\lambda$ of the form $\lambda=(b/k,b/k,\dots ,b/k,0,0,\dots)$ have been dealt with in \cite{QT}. This, in conjunction with the KOH Theorem, allowed the authors to give a beautiful proof, and then a generalization, of a conjecture on Fermions, coming from quantum physics. 

Notice that the partitions $\lambda$ studied in \cite{BQQW,QT} correspond precisely to the summands in the KOH identity involving only one nonconstant Gaussian polynomial. The goal of this note is to illustrate how $N$-modular diagrams, after  modifying their standard definition so as to suitably include rows of length zero, provide a nice combinatorial interpretation for any arbitrary term of the sum in the KOH identity. We present, using  different bijections, two main natural classes of modular diagrams of partitions, always contained inside an $a\times b$ rectangle, that have the KOH summands as their (hence symmetric and unimodal) generating functions. One of our results, which holds under some technical assumptions, yields a broad generalization of the theorems of \cite{BQQW} and \cite{QT}.

\section{Definitions and preliminary results}

Let us  briefly recall the main facts and definitions that are  needed in this note. Given a  partition $\lambda=(\lambda_1,\lambda_2,\dots )\vdash  n$, the  nonzero $\lambda_i$  are known as the \emph{parts} of $\lambda $. The number of parts of $\lambda $ is  its \emph{length}, denoted by $l(\lambda)$, which is of course finite. The \emph{multiplicity} of an integer $i\geq 1$ in $\lambda$, denoted by $m_i:=m_i(\lambda)$, is the number of parts of $\lambda$ equal to $i$. Then a partition is sometimes also written as $\lambda=(1^{m_1},2^{m_2},\cdots )$, where the parts of multiplicity zero are  omitted. Notice that, if $\lambda=(\lambda_1,\lambda_2,\dots )=(1^{m_1},2^{m_2},\cdots )\vdash  n$, then $n=\sum_{i\geq 1}im_i$ and $l(\lambda)=\sum_{i\geq 1}m_i$. 

A partition $\lambda$ can be represented geometrically by its {\em Ferrers} (or {\em Young}) {\em diagram}, that is, by a collection of cells, arranged in left-justified rows, with the $i$-th row containing exactly $\lambda_i$ cells. The \emph{conjugate} partition, $\lambda'=(\lambda'_1,\lambda'_2,\dots )$, of $\lambda$ is the partition whose Ferrers diagram is obtained from that of $\lambda$ by interchanging rows and columns. It immediately follows that $m_i(\lambda)=\lambda'_i-\lambda_{i+1}'$ for all $i\geq 1$. For instance,  $\lambda=(5,4,4,4,2,2,1)=(1^1,2^2,4^3,5^1)$ is a partition of 22 of length $l(\lambda)=7$, whose conjugate is $\lambda'=(7,6,4,4,1)$. 

Among the several possible choices, for an introduction to partition theory,  a survey of the  main results and techniques, or  the philosophy behind this remarkably broad field, see \cite{And,AE,Pak}, Section I.1 of \cite{Ma}, and Section 1.8 of \cite{St0}.

It is well known that the generating function for all partitions $\lambda$ contained inside an $a\times b$ rectangle --- that is, partitions  $\lambda$ such that $l(\lambda)\leq b$ and $\lambda_1 \leq a$ --- is  the Gaussian polynomial $\binom {a+b}{b}_q$. This polynomial is clearly symmetric with nonnegative coefficients, and it is also unimodal, as first shown combinatorially by K. O'Hara \cite{OH}. (Several other proofs, coming from different areas of mathematics, are known for the unimodality of Gaussian polynomials; see \cite{Proc,St5,St1980,Sy}.) J. Quinn \emph{et al.}, in \cite{BQQW} and \cite{QT}, enumerated partitions inside an $a\times b$ rectangle subject to some further restrictions. They proved the following two results:

\begin{theorem}[\cite{BQQW}]\label{bqqw}
$$F_{(b)}(q):=q^{b^2-b} \binom{a+2-b}{b}_q$$
is the generating function for all partitions $\lambda $  contained inside an $a\times b$ rectangle,  such that $\lambda_i - \lambda_{i+1}\geq 2$ for all $i\leq b-1$.
\end{theorem}

We thank an anonymous referee for pointing out to us that a standard combinatorial argument for the previous result (simply consisting of attaching the  even staircase partition $(2b-2, 2b-4, \dots,4, 2)$ of $b^2-b$ to an arbitrary partition contained inside an $(a-2b+2)\times b$ rectangle)  has essentially been known since Schur or MacMahon. In all fairness, the proof given in \cite{BQQW}, which was along the same lines, was equally simple.

\begin{theorem}[\cite{QT}]\label{qt}
Fix a positive integer $k$ dividing $b$. Then
$$F_{\left((b/k)^k\right)}(q):=q^{b^2/k-b} \binom{k(a+2)-{b(2-1/k)}}{{b/k}}_q$$
is the generating function for all partitions $\lambda $   contained inside an $a\times b$ rectangle, such that $\lambda_i - \lambda_{i+k}\geq 2$ for all $i\leq b-k$, and $\lambda_i - \lambda_{i+k-1}\leq 1$ for all $i\equiv 1$ (mod $k$).
\end{theorem}

Note that the generating function of Theorem \ref{bqqw} is the summand in the right-hand side of the KOH formula being contributed by the partition $\lambda=(b)$, while the generating function of Theorem \ref{qt} corresponds to the partition $\lambda=\left((b/k)^k\right)$. Of course, Theorem \ref{bqqw} is  the special case $k=1$ of Theorem \ref{qt}.

We now introduce a definition of an $N$-modular diagram. Our definition will slightly differ from the usual one, in that it also differentiates among  entries equal to zero of a partition. Such a refinement will be essential in our second main result, and will make the first result more elegant. 

\begin{definition}\label{ddd}
Fix a partition $\lambda$ of length $l(\lambda)$, and positive integers $k$ and $N$ such that $k\geq l(\lambda)$. An \emph{$N$-modular diagram of length $k$} is the Ferrers diagram of $\lambda$ to which \emph{a zeroth column} of length $k$ has been added, such that all cells are labeled with an integer between 1 and $N$, all but the rightmost cell of any row are labeled  $N$, and the entries of any column are weakly decreasing from top to bottom. 

If the integer $k$ is clear from the context, we will simply speak of an \emph{$N$-modular diagram}. In particular,  in the proofs of the main theorems of this paper, we will make a repeated use of modular diagrams of partitions contained inside suitable rectangles, and the lengths of the zeroth columns will always coincide with the heights of the rectangles. Finally, we will simply say that a row is labeled $i$ if its rightmost cell is labeled $i$.
\end{definition}

\begin{figure}[h!t]
\includegraphics[scale=0.60]{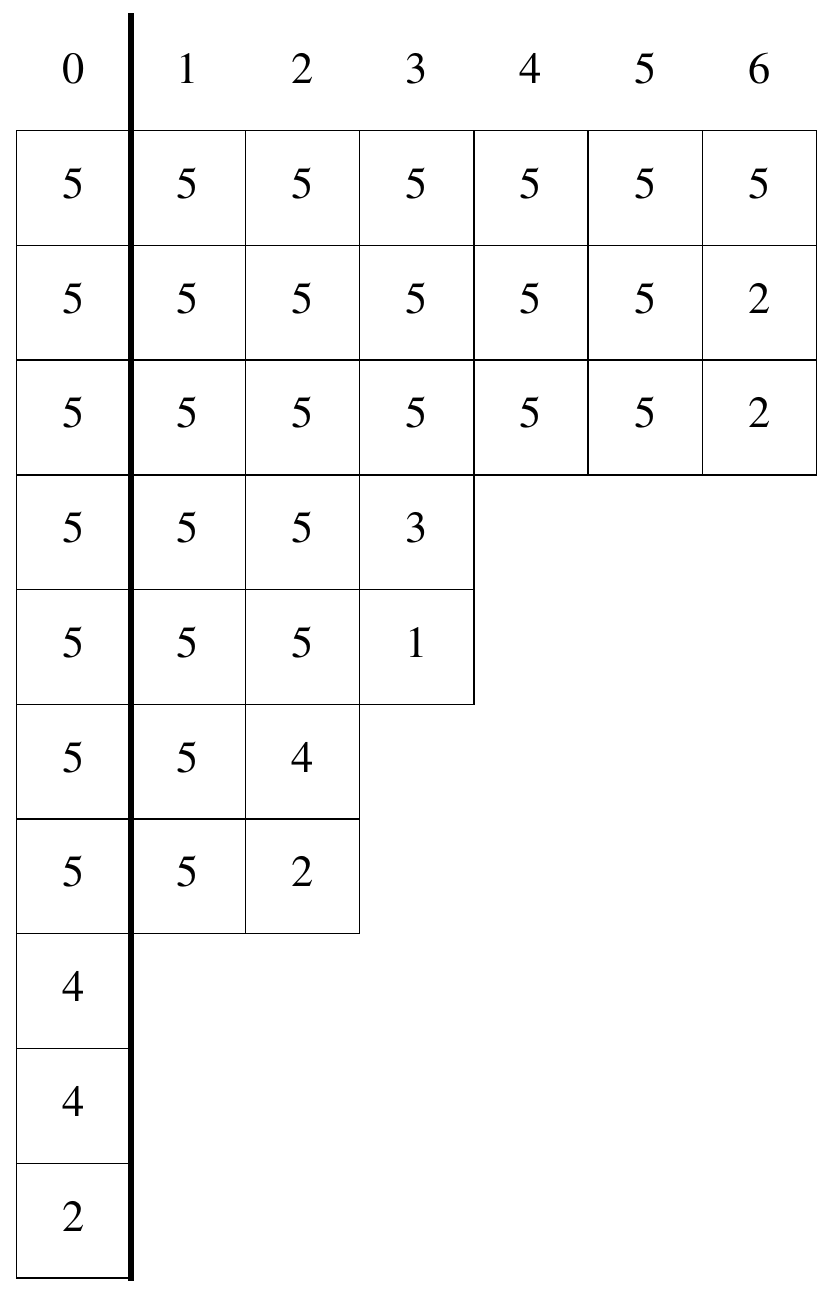} \caption{A 5-modular diagram of length 10 corresponding to $\lambda =(6,6,6,3,3,2,2)$.}
\end{figure}
Notice that, from Definition \ref{ddd}, an $N$-modular diagram must have the zeroth cell of any row labeled  $N$, except (possibly) if that row has length zero. Also, the last cells of the rows with the same length (including the rows having length zero) can be labeled with any integers between 1 and $N$, provided they weakly decrease from top to bottom. Finally, note that 2-modular diagrams can be put naturally  in bijection with \emph{MacMahon diagrams} (again, redefined with a \emph{zeroth column}), which have unmarked and marked cells instead of cells labeled with 2 and 1, respectively. See Figure 1 for one of the $\left(\binom{8}{3}-\binom{7}{2}\right)^2\cdot \left(\binom{7}{2}-\binom{6}{1}\right)^2=275,625$ possible 5-modular diagrams  of length 10 corresponding to the partition $\lambda =(6,6,6,3,3,2,2)$ of 28. 

\section{The main results}

Fix positive integers $a$ and $b$. Our object is to use modular diagrams, as we defined them above, to provide a natural combinatorial interpretation for the KOH summand,
$$F_\lambda(q):=q^{2\sum_{i\geq 1}\binom{\lambda_i}{2}} \prod_{j\geq 1} \binom{j(a+2)-Y_{j-1}-Y_{j+1}}{\lambda_j-\lambda_{j+1}}_q,$$
corresponding to any partition $\lambda\vdash b$.

We present two nice and essentially different classes of (hence symmetric and unimodal) modular diagrams of partitions, both contained inside our $a\times b$ rectangle, which have the $F_\lambda$ as their generating functions. One of these classes can be produced for any arbitrary partition $\lambda$, while the other requires some technical assumption (which we also indicate how to relax by introducing a third large, if less elegant, class of modular diagrams). Our second  modular diagrams provide a very natural (and broad) generalization of the above Theorems \ref{bqqw} and \ref{qt}.


The following lemma is a  known (and trivial) arithmetic fact, of which we omit the proof.

\begin{lemma}\label{unique}
Fix any two integers $c\geq 0$ and $d\geq 1$, and write $c=sd+r$, for the (unique) integers $s\geq 0$ and $0\leq r\leq d-1$. Then 
$$(c_1=s+1,\dots ,c_r=s+1,c_{r+1}=s,\dots ,c_d=s)$$
is the unique partition  $(c_1,\dots ,c_d)$ of $c$  such  that $c_1-c_d\leq 1$.
\end{lemma}

Our first main result is:

\begin{theorem}\label{a}
Let $\lambda$ be any arbitrary partition of $b$. Then $F_\lambda$ is the generating function for all $\lambda_1$-modular diagrams $\Lambda $ of length $b$ (contained inside an $a\times b$ rectangle) satisfying the following conditions:
\begin{enumerate}
\item  $\Lambda $ has $\lambda'_i$ rows labeled $i$, for each $i\geq 1$;
\item If we denote by $\Sigma_i$ the sum of  the lengths of all rows of $\Lambda $ labeled $i$, then for each positive integer $j$, we have:
$$\Sigma_{\lambda_{j+1}+1}\leq \Sigma_{\lambda_{j+1}+2}\leq \dots \leq \Sigma_{\lambda_j}\leq j(a+1+\lambda_j+\lambda_{j+1}) -2Y_j;$$
\item Any row labeled with an integer between $\lambda_{j+1}+1$ and $\lambda_j$ has length at least $\lambda_j +\lambda_{j+1}-1$;
\item The difference between the lengths of the largest and the smallest row labeled with the same integer is at most 1.	
\end{enumerate}
(Notice that conditions (2) and (3) are nonempty only when $\lambda_{j}>\lambda_{j+1}$.)
\end{theorem}

\begin{proof}
We want to construct, with a series of bijections, the $\lambda_1$-modular diagrams whose generating function is $F_\lambda$. Because of conditions (1) and (3) notice that, for each index $j$ such that $\lambda_{j}>\lambda_{j+1}$, we have
$$\Sigma_{\lambda_{j+1}+1}\geq \lambda'_{\lambda_{j+1}+1}(\lambda_j+\lambda_{j+1}-1)=j(\lambda_j+\lambda_{j+1}-1).$$
Therefore, from  (2) it easily follows that 
$$\alpha^j:=(\alpha_{\lambda_{j}}=\Sigma_{\lambda_{j}}-j(\lambda_j+\lambda_{j+1}-1), \alpha_{\lambda_{j}-1}=\Sigma_{\lambda_{j}-1}-j(\lambda_j+\lambda_{j+1}-1), $$
$$\dots ,\alpha_{\lambda_{j+1}+1}=\Sigma_{\lambda_{j+1}+1}-j(\lambda_j+\lambda_{j+1}-1))$$
is a partition  contained inside a $(j(a+2) -2Y_j) \times (\lambda_j-\lambda_{j+1})$ rectangle, say $R_j$; the generating function for all such partitions is of course given by the Gaussian polynomial
$$\binom{j(a+2)-2Y_j+(\lambda_j-\lambda_{j+1})}{\lambda_j-\lambda_{j+1}}_q=\binom{j(a+2)-Y_{j-1}-Y_{j+1}}{\lambda_j-\lambda_{j+1}}_q.$$

Thanks to Lemma \ref{unique}, we can partition uniquely all the $\sum_{j\geq 1}(\lambda_j-\lambda_{j+1})=\lambda_1$ integers $\alpha_i$ (coming from all of the corresponding indices $j$) as $\alpha_i=(x_1^i,x_2^i,\dots ,x_j^i)$, where $x_1^i-x_j^i\leq 1$. Now, for all $j$, $i$, and $\ell =1,2,\dots ,j$, set 
$$\lambda_{\ell }^i:= x_{\ell }^i+\lambda_j +\lambda_{j+1}-1.$$
We have that the total number of integers $\lambda_{\ell }^i$, for all $j$, $i$ and $\ell$, is $\sum_{j\geq 1}j(\lambda_j-\lambda_{j+1})=b$. Therefore, for each $i$ and $\ell $, let us label all integers $\lambda_{\ell }^i$ with  $i$. One moment's thought  shows that, by properly rearranging all the $\lambda_{\ell }^i$ according to their length and label, we obtain a unique $\lambda_1$-modular diagram of length $b$ having the $\lambda_{\ell }^i$ as its rows.

Notice that, by summing over all indices $j$, $i$ and $\ell $ defined above, the sum of the entries $\lambda_j +\lambda_{j+1}-1$ coming from all the $\lambda_{\ell }^i$ is
$$\sum_{i\geq 1}i(\lambda_i-\lambda_{i+1})(\lambda_i+\lambda_{i+1}-1)=\sum_{i\geq 1}i(\lambda_i^2-\lambda_{i+1}^2)-i(\lambda_i-\lambda_{i+1})$$$$=\sum_{i\geq 1}\lambda_i^2-\lambda_i=2\sum_{i\geq 1}\binom{\lambda_i}{2}.$$

Also,  each of the $j$ partitions $\alpha^j$ above can be chosen independently inside a rectangle $R_j$. Therefore it easily follows that the above $\lambda_1$-modular diagrams are  enumerated by the generating function  $F_\lambda$. We leave to the reader the standard task of verifying that all steps of our construction are reversible. This concludes the proof of the theorem.
\end{proof}

\begin{remark}
The $\lambda_1$-modular diagrams constructed in the proof of Theorem \ref{a} are, in fact, all contained inside an $a\times b$ rectangle. Indeed, by condition (1) of the theorem, they contain exactly $\sum_{i\geq 1}\lambda'_i=\sum_{i\geq 1}\lambda_i=b$ (nonnegative) rows. Hence, if $\left\lceil n/d\right\rceil $ as usual denotes the smallest integer $\geq n/d$, by Lemma \ref{unique} and condition (2), it suffices to check  that, for all indices $j$ such that $\lambda_j>\lambda_{j+1}$,
\begin{equation}\label{inside}
\left\lceil \frac{j(a+1+\lambda_j+\lambda_{j+1})-2Y_j}{j}\right\rceil \leq a.
\end{equation}
But, clearly, $Y_j\geq j\lambda_j$. Thus, the left-hand side of inequality (\ref{inside}) is bounded from above by
$$\left\lceil \frac{j(a+1+\lambda_j+\lambda_{j+1})-2j\lambda_j}{j}\right\rceil=a+1-\lambda_j+\lambda_{j+1},$$
which is $\leq a$ since $\lambda_j>\lambda_{j+1}$, as desired.
\end{remark}

We illustrate the idea of the proof of Theorem \ref{a} with the following example.
\begin{figure}[h!t]
\includegraphics[scale=0.9]{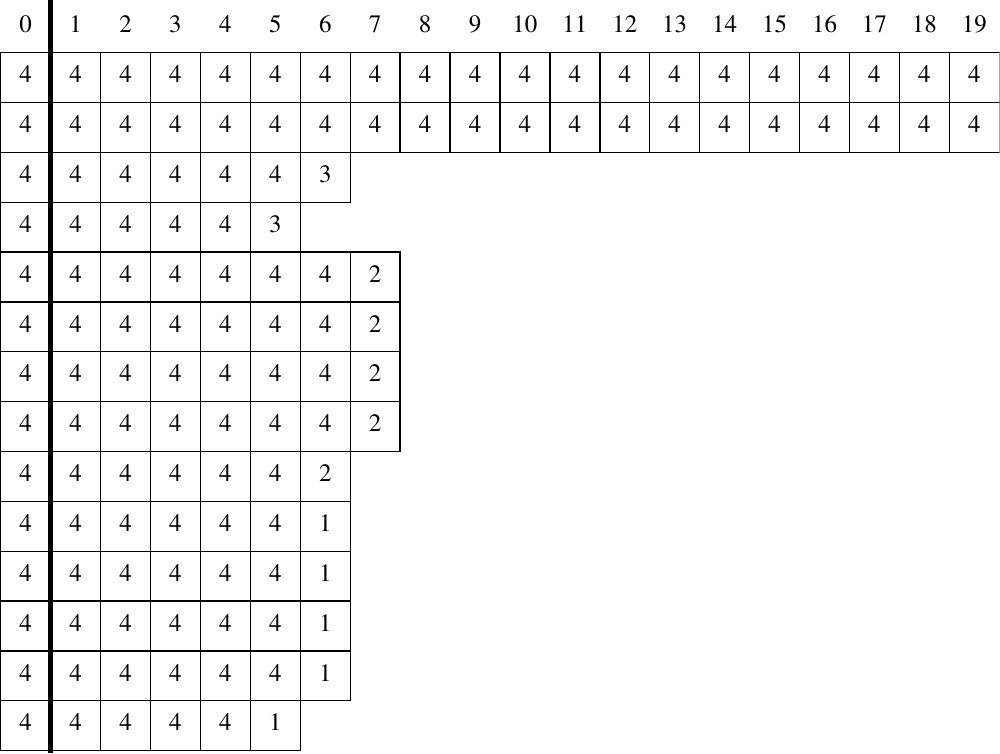}
\includegraphics[scale=0.9]{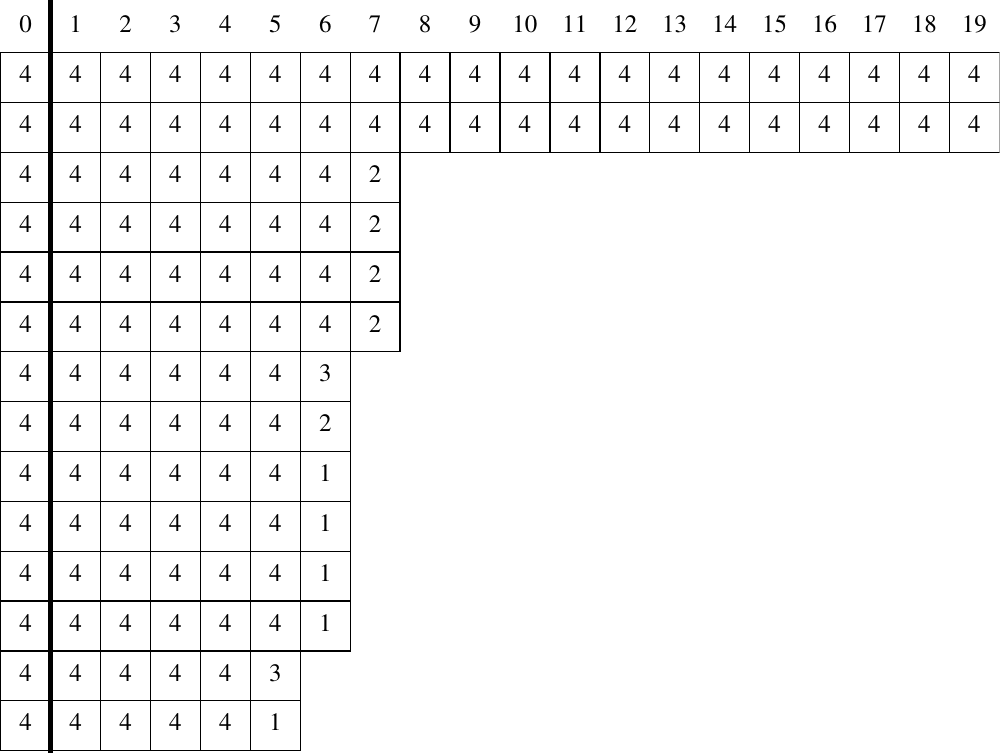} \caption{The 4-modular diagram of Example \ref{33}, before (above) and after (below) rearranging its rows according to  length and label.}
\end{figure}

\begin{example}\label{33}
Let  $a=20$, $b=14$, and $\lambda=(4^2,2^3)$. From Theorem \ref{a}, we can compute that the $\Sigma_i$ must satisfy the inequalities $10\leq \Sigma_3\leq \Sigma_4\leq 38$, and $5\leq \Sigma_1\leq \Sigma_2\leq 87$. Thus, we may freely choose any partition $\alpha^1=(\alpha_4,\alpha_3)$, whose entries  will in turn be partitioned to (eventually) give the rows labeled 4 and 3 of our 4-modular diagram of length $b=14$, inside a $28\times 2$ rectangle; and any partition $\alpha^2=(\alpha_2,\alpha_1)$, eventually yielding  the rows labeled 2 and 1, inside an $82\times 2$ rectangle.

Pick for instance $\alpha^1=(\alpha_4,\alpha_3)=(28,1)$ and $\alpha^2=(\alpha_2,\alpha_1)=(29,24)$. We want to construct the corresponding 4-modular diagram. By Lemma \ref{unique}, the $\alpha_i$ partition as follows: 
$$\alpha_1=(x_1^1,\dots ,x_5^1)=(5^4,4),$$
$$\alpha_2=(x_1^2,\dots ,x_5^2)=(6^4,5),$$ 
$$\alpha_3=(x_1^3,x_2^3)=(1,0),$$
$$\alpha_4=(x_1^4,x_2^4)=(14^2).$$

Therefore, the rows of our (eventual) 4-modular diagram are $\lambda_1^1=\lambda_2^1=\lambda_3^1=\lambda_4^1=6$ and $\lambda_5^1=5$ labeled 1; $\lambda_1^2=\lambda_2^2=\lambda_3^2=\lambda_4^2=7$ and $\lambda_5^2=6$ labeled 2; $\lambda_1^3=6$ and $\lambda_2^3=5$ labeled 3; and $\lambda_1^4=\lambda_2^4=19$ labeled 4. Rearranging them according to their length and label, we  uniquely determine the desired 4-modular diagram of length 14. (See Figure 2.)
\end{example}

As a very special case, Theorem \ref{a} recovers the generating function $F_{\left((b/k)^k\right)}$ of Theorem \ref{qt}, but with a different combinatorial interpretation. Namely, we have:

\begin{corollary}
Fix any positive  integer $k$ dividing $b$. Then 
$$q^{b^2/k-b} \binom{k(a+2)-b(2-1/k)}{{b/k}}_q$$
is the generating function for all $(b/k)$-modular diagrams $\Lambda $ of length $b$ (contained inside an $a\times b$ rectangle) satisfying the following conditions:
\begin{enumerate}
\item  $\Lambda $ has $k$ rows labeled $i$, for each $i=1,2,\dots ,b/k$;
\item If we denote by $\Sigma_i$ the sum of  the lengths of all rows of $\Lambda $ labeled $i$, then
$$\Sigma_{1}\leq \Sigma_{2}\leq \dots \leq \Sigma_{{b/k}}\leq k(a+1)-b;$$
\item All rows of $\Lambda $ have length at least $b/k-1$;
\item The difference between the lengths of the largest and the smallest row labeled with the same integer is at most 1.	
\end{enumerate}
\end{corollary}

\begin{proof}
This is simply the  case $\lambda=\left((b/k)^k\right)$ of Theorem \ref{a}.
\end{proof}

As another application  of Theorem \ref{a},  the following nice class of symmetric and unimodal MacMahon diagrams (always with nonnegative rows) can be  constructed as a special case:

\begin{corollary}
Let $2m_2+m_1=b$. Then
$$\frac{q^{3m_2+m_1}(1-q^{m_2(a-2)+1})(1-q^{m_2(a-2)+am_1+1})}{(1-q)^2}$$ 
is the generating function for all MacMahon diagrams $\Lambda $ of length $b$ (contained inside an $a\times b$ rectangle) satisfying the following conditions:
\begin{enumerate}
\item   $\Lambda $ has $m_2$ unmarked rows  and $m_1+m_2$ marked rows;
\item The sum of  the lengths of all marked rows of $\Lambda $ is at most $m_2(a-2)+am_1$,  and the sum of  the lengths of all unmarked rows  is at most $am_2$;
\item Any unmarked row has length at least 2;
\item The difference  between the lengths of  the largest and the smallest marked (resp., unmarked) row is at most 1.	
\end{enumerate}
\end{corollary}

\begin{proof}
A standard computation shows that this is the  case $\lambda=(2^{m_2},1^{m_1})$ of Theorem \ref{a}, where we replace all cells labeled 1 with marked cells and all cells labeled 2 with unmarked cells.
\end{proof}

We now present the second main theorem of this note, which generalizes the results, as well as the bijections, of \cite{BQQW} and \cite{QT}. We have:

\begin{theorem}\label{b}
Let $\lambda $ be a partition of $b$. For each index $d$, suppose  there exists a partition $\gamma^d=(\gamma_1^d,\gamma_2^d, \dots , \gamma_{\lambda_d-\lambda_{d+1}}^d)$ having distinct parts, such that  $\gamma_{\lambda_d-\lambda_{d+1}}^d=0$, $$\gamma_1^d=-2+2Y_d/d,$$
 and 
 $$\sum_{h,d}d\gamma_h^d =2\sum_{i\geq 1}\binom{\lambda_i}{2}.$$
(Notice that the partitions $\gamma^d$ are nonzero only when $\lambda_{d}>\lambda_{d+1}$.) Then, if $p$ is the number of distinct part sizes of $\lambda$, for any fixed $p$-tuple $(\gamma^{d_1},\gamma^{d_2},\dots ,\gamma^{d_p})$  of such nonzero partitions, where $d_1>d_2> \dots >d_p$, $F_\lambda$ is the generating function for all $p$-modular diagrams $\Theta $ of length $b$, contained inside an $a\times b$ rectangle,  such that, for each $j=1,2,\dots ,p$, $\Theta $  has ${d_j}(\lambda_{d_j}-\lambda_{{d_j}+1})$ rows labeled $j$, say
$$\Theta_1^j\geq \Theta_2^j\geq \dots \geq \Theta_{{d_j}(\lambda_{d_j}-\lambda_{{d_j}+1})}^j,$$ 
satisfying the following two conditions:
\begin{enumerate}
\item  $$\Theta_{i{d_j}+1}^j-\Theta_{{d_j}(i+1)}^j\leq 1,$$
for all $i=0,1,\dots ,\lambda_{d_j}-\lambda_{{d_j}+1}-1$;
\item $$\Theta_{c{d_j}+h}^j-\Theta_{{d_j}(c+1)+h}^j\geq \gamma_{c+1}^{d_j} -\gamma_{c+2}^{d_j},$$
for all $c=0,1,\dots ,\lambda_{d_j}-\lambda_{{d_j}+1}-2$ and all $h=1,2,\dots ,{d_j}$.
\end{enumerate}
\end{theorem}

\begin{proof}
Fix $\lambda$ and the $p$ partitions $\gamma^{d_j}$ as in the statement. The main idea will be to generalize the Quinn-Tobiska bijections used to prove Theorem \ref{qt}, by means of our modular diagrams with nonnegative rows. Start by fixing any index $j=1,2,\dots ,p$. We want to construct bijectively the rows $\Theta_i^j$ labeled $j$ of our eventual $p$-modular diagram of length $b$.

Consider any partition
$$\beta^j:=\left(\beta_1^j,\beta_2^j,\dots ,\beta_{\lambda_{d_j}-\lambda_{{d_j}+1}}^j\right)$$
contained inside a ${d_j}(a-\gamma_1^{d_j})\times (\lambda_{d_j}-\lambda_{{d_j}+1})$ rectangle, say $S_j$. Define then a new partition $$\rho^j:=\left(\rho_1^j,\rho_2^j,\dots ,\rho_{\lambda_{d_j}-\lambda_{{d_j}+1}}^j\right),$$
where, for each $i=1,2,\dots ,\lambda_{d_j}-\lambda_{{d_j}+1}$, we set
$$\rho_i^j:=\beta_i^j+{d_j}\gamma_i^{d_j}.$$

Now, for each index $i$, thanks to Lemma \ref{unique}, we can in turn uniquely partition the integer $\rho_i^j$  as
$$\rho_i^j:=\left(\Theta_{{d_j}(i-1)+1}^j,\Theta_{{d_j}(i-1)+2}^j,\dots ,\Theta_{i{d_j}}^j\right).$$

We claim that these $\Theta_t^j$, for $t=1,2,\dots ,{d_j}(\lambda_{d_j}-\lambda_{{d_j}+1})$, are the rows labeled $j$ of our eventual $p$-modular diagram.

Indeed, we have $\Theta_1^j\geq \Theta_2^j\geq \dots \geq \Theta_{{d_j}(\lambda_j-\lambda_{j+1})}^j$, for the integers $\rho_i^j$ are different for all $i$ (since,  by hypothesis, $\gamma_{h}^{d_j} >\gamma_{h+1}^{d_j}$ for all $h$). Condition (1) of the statement is obviously satisfied by construction. Finally, we have $\rho_i^j-\rho_{i+1}^j\geq  {d_j}(\gamma_i^{d_j}-\gamma_{i+1}^{d_j})$, which, by dividing by ${d_j}$ and using Lemma \ref{unique}, is easily seen to imply condition (2) of the statement.

Furthermore, similarly to how we argued in the proof of Theorem \ref{a}, since we are considering our modular diagrams, all $j$ partitions $\beta^j$ can be chosen independently of one another inside their rectangles $S_j$. Therefore, the assumption $\gamma_1^{d_j}=-2+2Y_{d_j}/{d_j}$, which implies  that the sum of the two sides of the rectangle $S_j$ is
$${d_j}(a-\gamma_1^{d_j})+ (\lambda_{d_j}-\lambda_{{d_j}+1})= d_j(a+2)-Y_{d_j-1}-Y_{d_j+1},$$
easily gives that the generating function for all partitions constructed above is $F_\lambda$, as desired. 

That all steps we have performed in this proof are reversible is a standard fact to check, and follows the idea of the proof of Theorem \ref{qt} (see \cite{QT}), hence will be omitted.
\end{proof}

Let us illustrate the argument of Theorem \ref{b} with an example.

\begin{example}\label{39}
Fix the integers $a=15$ and $b=20$, and consider the partition $\lambda=(7,7,2,2,2)$ of $b$. Since $\lambda_2>\lambda_3$ and $\lambda_5>\lambda_6$, the two nonzero partitions $\gamma^d$ are $\gamma^2=(\gamma_1^2,\dots , \gamma_{5}^2)$ and $\gamma^5=(\gamma_1^5, \gamma_{2}^5)$. It is easy to check that we have $\gamma_1^2=12$, $\gamma_1^5=6$, and $\gamma_5^2=\gamma_2^5=0$. Let us pick $\gamma_2^2=10$, $\gamma_3^2=4$ and $\gamma_4^2=3$. 

Hence we may freely choose the partition $\beta^1 =(\beta_1^1, \beta_2^1, \beta_3^1, \beta_4^1, \beta_5^1)$ inside a $6\times 5$ rectangle, and the partition $\beta^2 =(\beta_1^2, \beta_2^2)$ inside a $45\times 2$ rectangle. Let us pick for instance $\beta^1 =(5,5,4,3,1)$ and $\beta^1 =(44,1)$, and construct the corresponding 2-modular diagram (or equivalently, the corresponding MacMahon diagram) of length $b=20$.

We have that the partitions $\rho^j$ we obtain from the $\beta^j$ are $\rho^1=(29,25,12,9,1)$ and $\rho^2=(74,1)$. Hence the rows labeled 1 of our eventual modular diagram, given by partitioning the entries of $\rho^1$ according to Lemma \ref{unique}, are:
$$\Theta_1^1=15,{\ }\Theta_2^1=14,{\ }\Theta_3^1=13,{\ }\Theta_4^1=12,{\ }\Theta_5^1=\Theta_6^1=6,{\ }\Theta_7^1=5,{\ }\Theta_8^1=4,{\ }\Theta_9^1=1,{\ }\Theta_{10}^1=0;$$
the rows labeled 2, obtained by partitioning the entries of $\rho^2$, are:
$$\Theta_1^2=\Theta_2^2=\Theta_3^2=\Theta_4^2=15, {\ }\Theta_5^2=14, {\ }\Theta_6^2=1, {\ }\Theta_7^2=\Theta_8^2=\Theta_9^2=\Theta_{10}^2=0.$$
Finally, by rearranging all the $\Theta_i^j$ according to their length and label, we uniquely determine our 2-modular diagram. (See Figure 3.)
\end{example}

\begin{remark}
\begin{enumerate}
\item Notice that one condition of Theorem \ref{b} implies the restriction on $\lambda$ that ${d_j}$ must divide $2Y_{d_j}$, for all $j$. Also, it is easy to see that, except in the degenerate case $\lambda=(1,1,\dots ,1)$,  we have $\gamma_1^{d_j}>0$ for each $j$. In particular, $\gamma_{\lambda_{d_j}-\lambda_{{d_j}+1}}^{d_j}\neq \gamma_1^{d_j}$, and therefore the parts of $\lambda$ that are not equal must differ by at least 2.

\item Our modified definition  of a modular diagram is necessary in the proof of Theorem \ref{b}. Indeed, for any $j$, the last entry $\rho_{\lambda_{d_j}-\lambda_{{d_j}+1}}^j$ of $\rho^j$ can be chosen to be small enough (in particular, any integer between 0 and ${d_j}-1$); that is, for suitable choices of the partitions $\beta^j$, the smallest values of the $\Theta_i^j$ can be simultaneously zero for more than one $j$. Therefore, in order to preserve the bijectivity of our maps and the conclusion of the theorem, we need also label and order the rows of length zero of $\Theta $.
\end{enumerate}
\end{remark}
\begin{center}
\begin{tabular}{cc}
\includegraphics[scale=0.8]{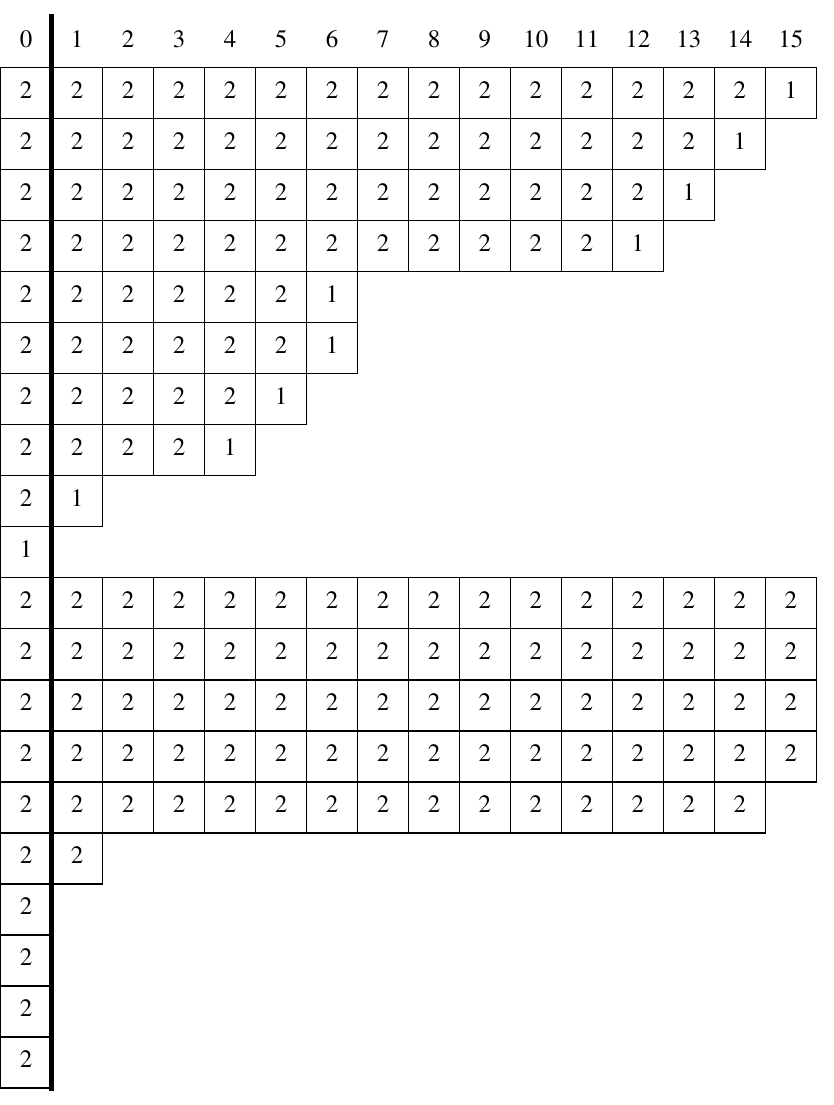}&
\includegraphics[scale=0.8]{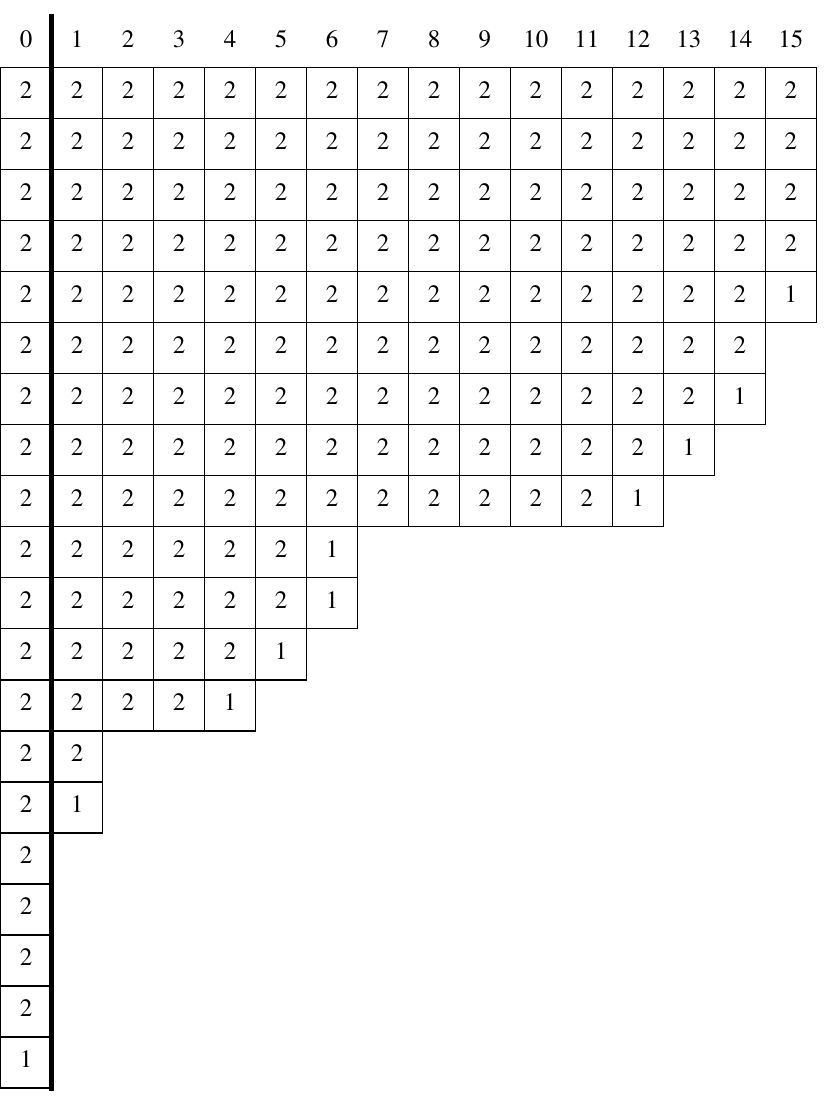}\\
(a)&
(b)\\
\end{tabular}
\end{center}
\begin{tabular}{cc}
\includegraphics[scale=0.8]{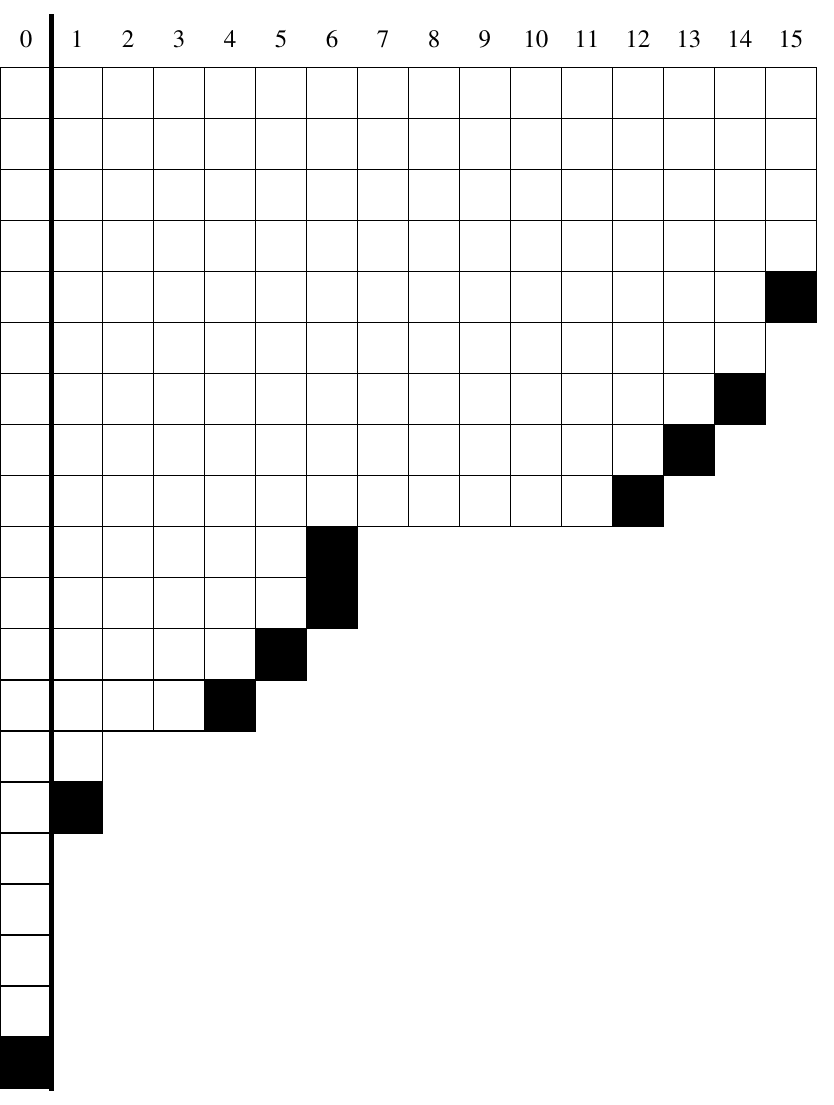}
(c)\\{\ }\\Figure 3. The 2-modular diagram of Example \ref{39}, (a) before and (b) after \\rearranging its rows, and (c) its corresponding MacMahon diagram.
\end{tabular}
{\ }\\
\\
\\\indent
The following remark shows how, by relaxing one of the assumptions of Theorem \ref{b} (but losing some elegance), we can obtain yet another broad class of modular diagrams with generating function $F_\lambda$.

\begin{remark}
The assumption that the partitions $\gamma^{d_j}$ be chosen to have distinct parts is necessary in the proof of Theorem \ref{b} in order to construct a $p$-modular diagram (instead of another $\lambda_1$-modular diagram). Indeed, if $\gamma_i^{d_j}=\gamma_{i+1}^{d_j}$ for some $j$ and $i$, we might lose the condition that the integers $\Theta_i^j$ be weakly decreasing for each given $j$. 

For instance, suppose that  the two largest parts of $\lambda$ both occur with multiplicity one, and that we may choose, say, $\rho_1^2=\rho_{2}^2=5$. Hence $\rho_1^2=(\Theta_{1}^2,\Theta_{2}^2)=(3,2)$ and $\rho_2^2=(\Theta_{3}^2,\Theta_{4}^2)=(3,2)$, giving the contradiction $\Theta_{2}^2<\Theta_{3}^2$. Notice that simply reordering the $\Theta_i^2$ would not suffice to save the bijectivity of the construction. Indeed, reordering  the above $\Theta_i^2$ gives the tuple $(3,3,2,2)$, but this  also corresponds to the values of $\Theta_i^2$ partitioning the integers $\rho_1^2=6$ and $\rho_{2}^2=4$.

We just mention here that, in fact, the assumptions on the integers $\gamma_i^{d_j}$ can be considerably relaxed, provided we suitably modify the conclusions of Theorem \ref{b} and, in particular,  consider again $\lambda_1$-modular diagrams   as opposed to our class of modular diagrams with only $p$ labels. The essential difference in this new construction is to require that each integer $\rho_i^j$ be partitioned into entries having a different label for  different $j$ \emph{and also for  different $i$}. We omit the exact statement of this (less elegant) alternative form of Theorem \ref{b}, whose argument and conclusions are  closer to those of Theorem \ref{a}.
\end{remark}

Note that, interestingly,  any special case of Theorem \ref{b} corresponding to $\lambda=\left((b/k)^k\right)$  provides a  symmetric and unimodal class of ordinary partitions contained inside an $a\times b$ rectangle with generating function $F_{\left((b/k)^k\right)}$. Namely, we have:

\begin{corollary}\label{int}
Fix any integer $k$ dividing $b$, and any partition $\gamma=\left(\gamma_1,\gamma_2, \dots , \gamma_{b/k}\right)$ having distinct parts, such that $\gamma_{b/k}=0$, $\gamma_1=2\left(b/k-1\right) $,  and
$$\sum_{h=1}^{{b/k}}\gamma_h =(b/k)^2-b/k.$$
Then, for any such given partition $\gamma $, 
$$ q^{b^2/k-b} \binom{k(a+2)-{b(2-1/k)}}{{b/k}}_q$$
is the generating function for all  partitions $\Theta$, contained inside an $a\times b$ rectangle, satisfying the following two conditions:
\begin{enumerate}
\item  $\Theta_{ik+1}-\Theta_{(i+1)k}\leq 1$, for all $i=0,1,\dots ,b/k-1$;
\item $\Theta_{ck+h}-\Theta_{(c+1)k+h}\geq \gamma_{c+1} -\gamma_{c+2},$ for all $c=0,1,\dots ,b/k-2$ and all $h=1,2,\dots ,k$.
\end{enumerate}
\end{corollary}

\begin{proof}
This is the special case of Theorem \ref{b} where $\lambda=\left((b/k)^k\right)$ and $\gamma_i^{k} =\gamma_i$. Since 1-modular diagrams are in obvious bijection with ordinary Ferrers diagrams and therefore with partitions, the result immediately follows.
\end{proof}

In particular, by choosing the $\gamma_i$ to have constant difference for all $i$, we  obtain, as a further special case, Quinn-Tobiska's Theorem \ref{qt}:

\begin{proof}
In Corollary \ref{int}, set  $\gamma_i-\gamma_{i+1}=2$  for all $i=1,2,\dots ,b/k-1$.
\end{proof}

As a final illustration, the following  class of symmetric and unimodal MacMahon diagrams can be constructed as a (very) special case of  Theorem \ref{b}.

\begin{corollary}
Let $tm_t+sm_s=b$, for some integers $t>s+1$ and $s>1$. Suppose that $\lambda'_1(s-1)$ divide $2(b-\lambda'_1)$, and that $t-s-1$ divide $2(t-1)$. Set $$A:=\frac{2(t-1)}{t-s-1}, {\ }{\ }{\ } B:=\frac{2(b-\lambda'_1)}{\lambda'_1(s-1)}.$$
Then $F_\lambda$ is the generating function for all MacMahon diagrams $\Upsilon $ of length $b$, contained inside an $a\times b$ rectangle, satisfying the following two conditions:
\begin{enumerate}
\item  $\Upsilon $ has $s\lambda'_1$ unmarked rows, say $\mu_1\geq \mu_2\geq \dots \geq \mu_{s\lambda'_1}$, such that $$\mu_{i\lambda'_1+1}-\mu_{\lambda'_1(i+1)}\leq 1$$
for all $i=0,1,\dots ,s-1$, and
$$\mu_i-\mu_{i+\lambda'_1}\geq B$$
for all $i=1,2,\dots ,\lambda'_1(s-1)$;\\
\item $\Upsilon $ has $m_t(t-s)$ marked rows, say $\nu_1\geq \nu_2\geq \dots \geq \nu_{m_t(t-s)}$, such that
$$\nu_{im_t+1}-\nu_{m_t(i+1)}\leq 1$$
for all $i=0,1,\dots ,t-s-1$, and
$$\nu_i-\nu_{i+m_t}\geq A$$
for all $i=1,2,\dots ,m_t(t-s-1)$.
\end{enumerate}
\end{corollary}

\begin{proof}
A standard computation shows this is the particular case of Theorem \ref{b} corresponding to $\lambda=(t^{m_t},s^{m_s})$, $\gamma_i^{m_t}-\gamma_{i+1}^{m_t}=A$  and $\gamma_i^{\lambda'_1}-\gamma_{i+1}^{\lambda'_1}=B$  for all $i$, and $\Theta_{\ell }^2=\mu_{\ell }$ and $\Theta_{\ell }^1=\nu_{\ell }$ for all $\ell $.
\end{proof}

\begin{remark}
It would be interesting  to determine a significant application to quantum physics of our results, so to also generalize  the applications presented in \cite{BQQW}  and  \cite{QT}. In order to do this, it might be useful to find a good combinatorial explanation for $\binom{a+b}{b}_q- F_\lambda(q)$ in terms of modular diagrams, for any partition $\lambda$ of $b$.
\end{remark}

\section*{Acknowledgments}\label{sec:acknowledgements} I warmly thank Richard Stanley for his terrific hospitality this year and for his encouragement and inspiration. It is thanks to him if I have extended my research interests to enumerative combinatorics. I also wish to thank the MIT Math Department for partial financial support, and Dr. Mark Gockenbach and the Michigan Tech Math Department, from which I am on partial leave, for extra summer support. I am  grateful to Jennifer Quinn for sending me an offprint of her paper \cite{QT}, and to David Clark, a finishing Ph.D. student in combinatorics at Michigan Tech, for  producing the figures included in this paper. I also thank an anonymous editor of JCTA for spotting a typo in the crucial formula, and the three referees for several comments that helped improve the presentation of this paper.


\end{document}